\documentclass[12pt]{amsart}
\usepackage{amsmath, amssymb, amsfonts}
\usepackage{mathrsfs}
\usepackage{graphicx}
\usepackage{geometry}
\usepackage{amsthm}
\usepackage{hyperref}
\usepackage{indentfirst}
\numberwithin{equation}{section}

\newtheorem{theorem}{Theorem}[section]
\newtheorem{proposition}[theorem]{Proposition}
\newtheorem{lemma}[theorem]{Lemma}
\newtheorem{corollary}[theorem]{Corollary}

\newcommand{\gl}{\mathfrak{gl}}

\DeclareMathOperator{\qdet}{qdet}
\DeclareMathOperator{\sgn}{sgn}

\begin{document}
	
	\title[Central Elements and Determinantal Identities in the Elliptic Quantum Algebra \( \mathcal{A}_{q,p}(\widehat{\mathfrak{gl}}_N)\)]{Central Elements and Determinantal Identities in the Elliptic Quantum Algebra \( \mathcal{A}_{q,p}(\widehat{\mathfrak{gl}}_N) \)}
	\author{Yingjie Hu}
	\address{School of Mathematics and Statistics,  Central China Normal University, Wuhan, Hubei 430079, China}
	\email{yjhu@ccnu.edu.cn}
	\author{Zheng Li}
	\address{School of Artificial Intelligence, Jianghan University, Wuhan, Hubei 430056, China}
	\email{lz1994@jhun.edu.cn}
	\author{Jian Zhang}
	\address{School of Mathematics and Statistics,  Central China Normal University, Wuhan, Hubei 430079, China}
	\email{jzhang@ccnu.edu.cn}
	\thanks{{\scriptsize
			\hskip -0.6 true cm MSC (2020): Primary: 17B37; Secondary: 81R10.
	}}
	\maketitle
	
\begin{abstract}
Elliptic quantum algebra is the algebraic structure characterized by the elliptic solution of the Yang-Baxter equation.
In this paper, we construct a family of central elements \( \mathfrak{z}(z) \) for the elliptic quantum algebra \(\mathcal{A}_{q,p}(\widehat{\mathfrak{gl}}_{N})\) and show that they can be  expressed as quantum determinants, yielding an elliptic analogue of the Liouville formula. In addition, we establish determinantal identities, including Jacobi's ratio theorem and Sylvester's theorem.
\end{abstract}
	
\section{Introduction}
	
In analogy with the rational and trigonometric cases, which produce the Yangian and quantum affine algebras respectively, the elliptic solution of the Yang-Baxter equation discovered by
Baxter \cite{Baxter 1982} give rise to elliptic quantum  algebra.
These are divided into two classes, namely the vertex type and the face type, according to the corresponding elliptic solution. Vertex-type elliptic quantum algebras were introduced in \cite{FIJKMY 1994,FIJKMY 1995}, while the face-type counterparts were studied in \cite{JKOS 1999}.
Elliptic quantum algebras and their representation theory lie at the intersection of quantum integrable systems, quantum groups, and vertex operator algebras \cite{Belavin 1981,JKOS 1999}.

Unlike Yangians and quantum affine algebras carrying only one deformation parameter $q$, the vertex-type elliptic quantum algebras $\mathcal{A}_{q,p}(\widehat{\mathfrak{gl}}_N)$ depend on two independent deformation parameters $q$ and $p$. Their R-matrices incorporate Jacobi theta functions and infinite $q$-Pochhammer symbols.
A fundamental structural correspondence between $\mathcal{A}_{q,p}(\widehat{\mathfrak{gl}}_N)$ and its trigonometric counterpart $\rm{U}_q(\widehat{\mathfrak{gl}}_N)$ emerges in the non-elliptic limit $p \to 0$ \cite{FIR 2018}. Under this degeneration, the full elliptic $R$-matrix collapses to a twisted trigonometric solution $R'(z)$, which coincides with a deformed version of the principal-gradation universal $R$-matrix of $\rm{U}_q(\widehat{\mathfrak{gl}}_N)$. The map connecting $R'(z)$ to the untwisted principal trigonometric $R$-matrix is mediated by a constant diagonal universal twist $\mathcal{F}_{12}$, the same Drinfeld element that distorts the standard Hopf coproduct of $\rm{U}_q(\widehat{\mathfrak{gl}}_N)$ to produce the quasi-coproduct structure of the elliptic algebra. Concretely, this twist simultaneously reparametrises the Lax matrix generators of the quantum affine algebra, recasting them as the Lax operators appearing in the non-elliptic limiting realisation of $\mathcal{A}_{q,p}(\widehat{\mathfrak{gl}}_N)$. For further developments about the other type elliptic quantum algebras, see \cite{KK 2003,Kon 2009,Kon 2018}.

The center of \( \mathcal{A}_{q,p}(\widehat{\mathfrak{gl}}_N) \)  plays a key role in the separation of variables and the construction of eigenfunctions for integrable Hamiltonians.
For Yangians \( \rm{Y}(\mathfrak{gl}_N) \) and quantum affine algebras $\rm{U}_q(\widehat{\gl}_{N})$, the center can be constructed via quantum determinants using $R$ matrix  realization. Numerous determinantal identities  have been established over the years; see, for example, \cite{KS 1982,Molev 2007,MNO 1996}. In the elliptic quantum algebra \( \mathcal{A}_{q,p}(\widehat{\mathfrak{gl}}_N) \), quantum determinants were introduced by Frappat, Issing and Ragoucy \cite{FIR 2018} and they proved that the coefficients of quantum determinants belong to the center of \( \mathcal{A}_{q,p}(\widehat{\mathfrak{gl}}_N) \).
	For $N=2$, Avan, Frappat and Ragoucy \cite {AFR 2017} constructed another family of central elements which can be expressed in the quantum determinants
	by a Liouville-type theorem.	In this paper, we generalize the Liouville-type theorem to arbitrary $N$.  In addition, we establish determinantal identities, including Jacobi's ratio theorems and Sylvester's identities, extending the corresponding results for (twisted) Yangians and quantum affine (super-)algebras \cite{MNO 1996, Molev 2007,JLZ 2025,JZ 2024,JZ 2025}.
	
	Our findings demonstrate a structural parallel between
	between \( \mathcal{A}_{q,p}(\widehat{\mathfrak{gl}}_N) \) and Yangians, while also revealing new features due to the elliptic   \(R\)-matrix. They provide a systematic algebraic foundation for future studies of representation theory, integrable hierarchies, and deformed \( \mathcal{W}_N \)-algebras associated with elliptic quantum algebras.
	
	The plan of the paper is as follows. In Section \ref{sec:def}, we
	review the definition of \( \mathcal{A}_{q,p}(\widehat{\mathfrak{gl}}_N) \) via the \(RLL\) relations and the \(L\)-operator formalism,
	as well as basic properties of the elliptic \(R\)-matrix.
	In Section \ref{sec:Liouville}, we construct the power series \( \mathfrak{z}(z) \) whose coefficients are central  and prove that  \( \mathfrak{z}(z) \) coincides with the ratio of quantum determinants, yielding an elliptic analogue of the quantum Liouville formula. Section \ref{sec:minor} is devoted to determinantal identities, including the factorization formula for the quantum determinant, analogues of Jacobi's ratio theorem and Sylvester's theorem.
	
	\section{Prelimilaries}\label{sec:def}
	
	In this section, we recall some preliminaries on the $N$-elliptic $R$-matrix and the elliptic quantum algebra $\mathcal{A}_{q,p}(\widehat{\mathfrak{gl}}_N)$.
	
	\subsection{$N$-elliptic $R$-matrix}
	The $N$-elliptic $R$-matrix \cite{AFR 2017,AFR 2019,AFRS 1999} is defined as
	\[
	R(z)=\sum R_{a,c}^{b,d}(z) \, e_{a,b} \otimes e_{c,d},
	\] whose non-vanishing entries obey $a+c=b+d$, the addition of indices being understood modulo $N$.
	For $1 \le a,b,c,d \le N$, one defines
	\begin{equation}
		R_{a,c}^{b,d}(z) = \eta (z) S_{a,c}^b(z) \, \omega^{(a+c-b-d)/2} \, \delta_{a+c,b+d}^{(\text{mod}\;N)} \;,
	\end{equation}
	where
	\begin{equation}\label{RMatrixAN}
		S_{a,c}^b(z) = z^{\frac{2(b-a)}{N}} q^{\frac{2(c-b)}{N}} p^{\frac{(b-a)(c-b)}{N}}\frac{\Theta_{p^N}(p^{N+c-a} q^2 z^2)}{\Theta_{p^N}(p^{N+c-b} z^2)\Theta_{p^N}(p^{N+b-a} q^2)} \;.
	\end{equation}
	Here $\omega=e^{2i\pi/N}$, which implies $\omega^{(a+c-b-d)/2}$ equals $\pm 1$.
	This factor results from the gauge transformation done on the Belavin $R$-matrix that allows one to recover the original $\mathcal{A}_{q,p}(\widehat{\mathfrak{gl}}_2)$ matrix.
	
	The normalization coefficient $\eta(z)$ is given by
	\begin{equation}
		\eta(z) = \frac{z^{\frac{2}{N}}}{\kappa_N(z^2)} \frac{(p^N,p^N)^3_{\infty}}{(p,p)^3_{\infty}}\frac{\Theta_{p}(q^2)\Theta_{p}(p z^2)}{\Theta_{p}(q^2 z^2)}
\end{equation}
with
\begin{equation}
\frac{1}{\kappa_N(z^2)} =
\frac{(q^{2N} z^{-2};p,q^{2N})_{\infty} (q^{2} z^{2};p,q^{2N})_{\infty} (p z^{-2};p,q^{2N})_{\infty}
(p q^{2N-2} z^{2};p,q^{2N})_{\infty}}{(q^{2N} z^{2};p,q^{2N})_{\infty} (q^{2} z^{-2};p,q^{2N})_{\infty}
(p z^{2};p,q^{2N})_{\infty} (p q^{2N-2} z^{-2};p,q^{2N})_{\infty}} \;.
\label{kappa}
\end{equation}
The matrix $S$ is $\mathbb{Z}_N$-symmetric by construction, i.e. the coefficients satisfy $S_{a+n,c+n}^{b+n}(z)=S_{a,c}^b(z)$, $n=1,...,N$, where again the addition of indices is modulo $N$.
Here, $\Theta_p(z)$ denotes the Jacobi theta function defined by ($p \in \mathbb{C}$ such that $|p|<1$)
\begin{equation}\label{theta:prod}
\Theta_p(z) = (z;p)_\infty \, (pz^{-1};p)_\infty \, (p;p)_\infty
\end{equation}
and the infinite $q$-Pochhammer symbols are given by
\begin{equation}
(z;p_1,\dots,p_m)_\infty = \prod_{n_i \ge 0} (1-z p_1^{n_1} \dots p_m^{n_m}) \,.
\end{equation}
The Jacobi theta function enjoys the following properties:
\begin{align}\label{id:theta}
& \Theta_{a}(az) = \Theta_{a}(z^{-1})
\text{ and }
\Theta_{a}(a^nz) = \frac{(-1)^n}{z^n\,a^{n(n-1)/2}}\,\Theta_{a}(z)\,,\quad \forall\ n\in\mathbb{Z},  \\
& \Theta_{a^2}(az) = \Theta_{a^2}(az^{-1}).
\label{id:theta2}
\end{align}

	The $N$-elliptic $R$ matrices $\widehat{R}(z,p,q)\equiv \widehat{R}(z)$ is defined by
	\begin{equation}
		\widehat{R}(z) = \tau_N(q^{\frac{1}{2}} z^{-1}) R(z) \,,
		\label{RealR}
	\end{equation}
	where
	\begin{equation}
		\tau_N(z)=z^{\frac{2}{N}-2} \frac{\Theta_{q^{2N}}(q z^2)}{\Theta_{q^{2N}}(q z^{-2})} \,.
		\label{rescale}
	\end{equation}
	
	The $R$-matrix $R(z)$ satisfies the following properties \cite{AFRS 1999,RT 1986,Tracy 1985}:
	\begin{itemize}
		\item Yang-Baxter equation (also holds for $\widehat{R}(z)$):
		\[
		R_{12}\left(\frac{z_1}{z_2}\right) R_{13}\left(\frac{z_1}{z_3}\right) R_{23}\left(\frac{z_2}{z_3}\right)
		= R_{23}\left(\frac{z_2}{z_3}\right) R_{13}\left(\frac{z_1}{z_3}\right) R_{12}\left(\frac{z_1}{z_2}\right).
		\]
		
		\item Unitarity:
		\[
		R_{12}(z) R_{21}(z^{-1}) = \mathbb{I},
		\]
		
		\item Regularity ($P_{12} = \sum e_{i,j} \otimes e_{j,i}$ is the permutation matrix):
		\[
		R_{12}(1) = P_{12},
		\]
		
		\item Crossing-symmetry:
		\[
		R_{12}(z)^{t_2} R_{21}(z^{-1} q^{-N})^{t_2} = \mathbb{I},
		\]
		
		\item Antisymmetry:
		\[
		R_{12}(-z) = \omega \, (g^{-1} \otimes \mathbb{I}) \, R_{12}(z) \, (g \otimes \mathbb{I}),
		\]
		
		\item Quasi-periodicity:
		\[
		\widehat{R}_{12}(-z p^{\frac{1}{2}}) = (g^{\frac{1}{2}} h g^{\frac{1}{2}} \otimes \mathbb{I})^{-1} \widehat{R}_{21}(z^{-1})^{-1} (g^{\frac{1}{2}} h g^{\frac{1}{2}} \otimes \mathbb{I}),
		\]
		
		\item Invariance:
		\[
		(h \otimes h) R_{12}(z) = R_{12}(z) (h \otimes h).
		\]
	\end{itemize}
	
	The crossing-symmetry and the unitarity properties of $R_{12}$ allow one to exchange the inversion and the transposition when applied to the matrix $R_{12}$ (or to the matrix $\widehat{R}_{12}$). It provides a crossing-unitarity relation (also valid for $\widehat{R}$ due to the $q^N$-periodicity of the function $\tau_N$):
	\begin{equation}\label{2.8}
		\left(R_{12}(x)^{t_2}\right)^{-1} = \left(R_{12}(q^N x)^{-1}\right)^{t_2}.
	\end{equation}
	
	Note also that the unitarity property for $\widehat{R}_{12}$ reads
	\[
	\widehat{R}_{12}(z) \widehat{R}_{21}(z^{-1}) = \tau_N(q^{\frac{1}{2}} z) \tau_N(q^{\frac{1}{2}} z^{-1}) \equiv \mathcal{U}(z),
	\]
	where the function $\mathcal{U}(z)$ is defined as
	\[
	\mathcal{U}(z) = q^{\frac{2}{N}-2} \frac{\Theta_{q^{2N}}(q^2 z^2) \Theta_{q^{2N}}(q^2 z^{-2})}{\Theta_{q^{2N}}(z^2) \Theta_{q^{2N}}(z^{-2})}.
	\]
	
	\subsection{Elliptic quantum algebra $\mathcal{A}_{q,p}(\widehat{\mathfrak{gl}}_N)$}
	Let us give the definition of the elliptic quantum algebra
	$\mathcal{A}_{q,p}(\widehat{\mathfrak{gl}}_N)$ \cite{AFR 2017,AFR 2019,AFRS 1999}.
	Consider the free associative algebra generated by the operators
	$L_{ij}[n]$ where $i,j \in \mathbb{Z}_N$ and $n \in \mathbb{Z}$, and we define
	the formal series
	\begin{equation}
		L_{ij}(z) = \sum_{n \in \mathbb{Z}} L_{ij}[n] z^n ,
	\end{equation}
	encapsulated into a $N \times N$ matrix
	\begin{equation}
		L(z) =
		\begin{pmatrix}
			L_{11}(z) & \cdots & L_{1N}(z) \\
			\vdots    & \ddots & \vdots    \\
			L_{N1}(z) & \cdots & L_{NN}(z)
		\end{pmatrix}.
	\end{equation}
	One defines $\mathcal{A}_{q,p}(\widehat{\mathfrak{gl}}_N)$ by imposing the following constraints on
	$L_{ij}(z)$:
	\begin{equation}\label{2.3}
		\widehat{R}_{12}(\frac{z}{w})\, L_1(z)\, L_2(w)
		= L_2(w)\, L_1(z)\, \widehat{R}^*_{12}(\frac{z}{w}),
	\end{equation}
	where $L_1(z) \equiv L(z) \otimes \mathbb{I}$,
	$L_2(z) \equiv \mathbb{I} \otimes L(z)$,
	$\widehat{R}_{12}(z) \equiv \widehat{R}_{12}(z,q,p)$ is the $N$-elliptic $R$-matrix,
	and
	\[
	\widehat{R}^*(z) = \widehat{R}(z)\big|_{p \to p^*=pq^{-2c}}.
	\]
	
	It is useful to introduce the following two matrices:
	\begin{equation}
		L^+(z) = L\!\left(q^{\frac{c}{2}} z\right),
	\end{equation}
	\begin{equation}
		L^-(z) = \bigl(g^{\frac{1}{2}} h g^{\frac{1}{2}}\bigr)\,
		L\!\left(-p^{\frac{1}{2}} z\right)\,
		\bigl(g^{\frac{1}{2}} h g^{\frac{1}{2}}\bigr)^{-1}.
	\end{equation}
	Let $g$ and $h$ be the matrices of order $N$ defined by $g_{ij} = \omega^i \delta_{ij}$ and $h_{ij} = \delta_{i+1,j}$ for $1 \leq i,j \leq N$ with $\omega = e^{2i\pi/N}$, the addition of indices being understood modulo $N$. They obey coupled exchange relations following from \eqref{2.3}, periodicity condition and unitarity property of the matrices
	$\widehat{R}_{12}$ and $\widehat{R}^*_{12}$:
	\begin{equation}\label{2.6}
		\widehat{R}_{12}(\frac{z}{w})\, L_1^\pm(z)\, L_2^\pm(w)
		= L_2^\pm(w)\, L_1^\pm(z)\, \widehat{R}^*_{12}(\frac{z}{w}),
	\end{equation}
	\begin{equation}\label{2.7}
		\widehat{R}_{12}(q^{\frac{c}{2}} z/w)\, L_1^+(z)\, L_2^-(w)
		= L_2^-(w)\, L_1^+(z)\, \widehat{R}^*_{12}(q^{-\frac{c}{2}} z/w).
	\end{equation}
	

	\section{Liouville Formula}\label{sec:Liouville}
	
	In this section we will give a family of central elements in $\mathcal{A}_{q,p}(\widehat{\mathfrak{gl}}_N)$ and show that they are
	related to the quantum determinant by Liouville-type Formula.
	\begin{proposition}\label{prop:cent}
		There exist a series \( \mathfrak{z}(z) \) in z with coefficients in the algebra $\mathcal{A}_{q,p}(\widehat{\mathfrak{gl}}_N)$ such that
		\begin{equation}\label{3.2}
			L(z)^t (L(q^N z)^{-1})^t = \mathfrak{z}(z)\mathbf 1,
		\end{equation}
		and
		\begin{equation}\label{3.3}
			(L(q^N z)^{-1})^t L(z)^t =\mathfrak{z}(z)\mathbf 1,
		\end{equation}
		where $\mathbf 1$ denotes the identity matrix.
		Moreover, the coefficients of the series \( \mathfrak{z}(z) \) belong to the center of the algebra \( \mathcal{A}_{q,p}(\widehat{\mathfrak{gl}}_N) \).
	\end{proposition}
	
	\begin{proof}
		The proof is a analogue of Proposition 2.1 in \cite{JLM 2024}.
		Multiplying both sides of the defining relation equation \eqref{2.3} by \( L_2(w)^{-1} \) and applying transposition $t_2$, we derive the relation
		\[
		\widehat{R}_{12}\left( \frac{z}{w} \right)^{t_2} (L_2(w)^{-1})^t L_1(z) = L_1(z) (L_2(w)^{-1})^t \widehat{R}_{12}^*\left( \frac{z}{w} \right)^{t_2},
		\]
		and hence
		\begin{equation}\label{3.4}
			\left(\widehat{R}_{12}\!\left(\frac{z}{w}\right)^{t_2}\right)^{-1}
			L_1(z)(L_2(w)^{-1})^t
			=
			(L_2(w)^{-1})^t L_1(z)
			\left(\widehat{R}_{12}^*\!\left(\frac{z}{w}\right)^{t_2}\right)^{-1}.
		\end{equation}
		
		Using the property of the \( R \)-matrix \eqref{2.8}, which also holds for \( \widehat{R} \)
		\[
		\left( \widehat{R}_{12}(x)^{t_2} \right)^{-1} = \left( \widehat{R}_{12}(q^N x)^{-1} \right)^{t_2}.
		\]
		The matrix $\widehat{R}_{12}^*$ obeys also the unitarity, crossing-symmetry, antisymmetry and quasi-periodicity conditions, thus we have
		\[
		\left( \widehat{R}_{12}\left( \frac{q^N z}{w} \right)^{-1} \right)^{t_2} L_1(z) (L_2(w)^{-1})^t = (L_2(w)^{-1})^t L_1(z) \left( \widehat{R}_{12}^*\left( \frac{q^N z}{w} \right)^{-1} \right)^{t_2}.
		\]
		
		Moreover,
		\[
		R_{12}(1)= P, \quad
		\widehat{R}_{12}(1)= \tau_N(q^{-1/2})P, \quad
		\widehat{R}_{12}^*(1)
		= \widehat{R}_{12}(1)\big|_{p\to p^*}
		= \tau_N(q^{-1/2})P .
		\]
		Now, set \( w = q^N z \) and use the fact that \( P^{-1} = P \) and \( P^{t_2} = Q \), where \( Q = \sum e_{i,j} \otimes e_{i,j} \), we obtain
		\[
		(L_2(q^N z)^{-1})^t L_1(z) Q = Q L_1(z) (L_2(q^N z)^{-1})^t.
		\]
		
		Since \( Q \) has a one-dimensional image, both sides must be equal to \( Q \mathfrak{z}(z) \) for some series \( \mathfrak{z}(z) \) with coefficients in the elliptic quantum algebra \( \mathcal{A}_{q,p}(\widehat{\mathfrak{gl}}_N) \). Using the relations $QX_{1}=QX_{2}^t$ and $X_{1}Q=X_{2}^tQ$ which hold for an arbitrary matrix $X$, we can write the definition of \( \mathfrak{z}(z) \) as
		\begin{equation}
			(L_2(q^N z)^{-1})^t L_2(z)^t Q =Q\mathfrak{z}(z)
		\end{equation}
		and
		\begin{equation}
			Q L_2(z)^t (L_2(q^N z)^{-1})^t=Q\mathfrak{z}(z).
		\end{equation}
		By taking the trace over the first copy of \( \mathrm{End}\, \mathbb{C}^N \), we arrive at \eqref{3.2} and \eqref{3.3}.
		
		Then we use \eqref{3.2} and \eqref{3.3} to prove that the coefficients of $\mathfrak{z}(z)$ belong to the center of
		$\mathcal{A}_{q,p}(\widehat{\mathfrak{gl}}_N)$.	\begin{align*}
			L_1(w) \mathfrak{z}(z)
			&= L_1(w) (L_2(q^N z)^{-1})^t L_2(z)^t \\[6pt]
			&= \widehat{R}_{12}\!\left(\frac{w}{q^N z}\right)^{t_2}
			(L_2(q^N z)^{-1})^t L_1(w)
			\left(\widehat{R}_{12}^*\!\left(\frac{w}{q^N z}\right)^{t_2}\right)^{-1}
			L_2(z)^t \\[6pt]
			&= \widehat{R}_{12}\!\left(\frac{w}{q^N z}\right)^{t_2}
			(L_2(q^N z)^{-1})^t L_1(w)
			\left(\widehat{R}_{12}^*\!\left(\frac{w}{ z}\right)^{-1}\right)^{t_2}
			L_2(z)^t \\[6pt]
			&= \widehat{R}_{12}\!\left(\frac{w}{q^N z}\right)^{t_2}
			(L_2(q^N z)^{-1})^t L_2(z)^t
			\left(\widehat{R}_{12}\!\left(\frac{w}{z}\right)^{-1}\right)^{t_2}L_1(w) \\[6pt]
			&= \widehat{R}_{12}\!\left(\frac{w}{q^N z}\right)^{t_2}
			\left(\widehat{R}_{12}\!\left(\frac{w}{z}\right)^{-1}\right)^{t_2}
			\mathfrak{z}(z) L_1(w) \\[6pt]
			&= \mathfrak{z}(z) L_1(w).
		\end{align*}
		In the second and fourth equalities we used the $RLL$ relation \eqref{2.3}and \eqref{3.4}, in the third and sixth equalities we used the property of the $R$-matrix \eqref{2.8}, in the fifth we used \eqref{3.3}.
		Thus, $\mathfrak{z}(z)$ commutes with $L_1(w)$, showing that the coefficients of $\mathfrak{z}(z)$ are central elements in
		$\mathcal{A}_{q,p}(\widehat{\mathfrak{gl}}_N)$.
	\end{proof}
	
	\begin{corollary}
		We have the formulas
		\begin{equation}\label{3.7}
			\mathfrak{z}(z)=\frac{1}{N} tr(L(z)L(q^N z)^{-1})=\frac{1}{N} tr(L(q^N z)^{-1}L(z)).
		\end{equation}
	\end{corollary}
	
	\begin{proof}
		The formulas follow by taking trace on both sides of the respective matrix relations \eqref{3.2} and \eqref{3.3}.
	\end{proof}

The quantum determinant $\operatorname{qdet} L(z)$ of the elliptic quantum algebra \( \mathcal{A}_{q,p}(\widehat{\mathfrak{gl}}_N) \) is defined as
	\begin{equation}\label{qdet}
		L_1(z)\cdots L_N(zq^{1-N})A_N=\operatorname{qdet} L(z)A_N
	\end{equation}
	where $A_N$ is the antisymmetrizer of $(\mathbb{C}^N)^{\otimes N}$ \cite{FIR 2018}.
	
	Expanding \eqref{qdet}, we have that
	\begin{equation}\label{3.8}
		\operatorname{qdet} L(z)
		= \sum_{\sigma\in S_N} \sgn(\sigma)\,
		L_{1,\sigma(1)}(z)\,
		L_{2,\sigma(2)}(zq^{-1})\cdots
		L_{N,\sigma(N)}(zq^{1-N}).
	\end{equation}
	It is proved that the coefficients of $\operatorname{qdet} L(z)$ belong to the center of $\mathcal{A}_{q,p}(\widehat{\mathfrak{gl}}_N)$.
	In order to uncover the connection between the quantum determinant and the series $\mathfrak{z}(z)$ constructed in Proposition \ref{prop:cent}, we need the following lemma.
	\begin{lemma}
		The quantum comatrix $\widehat{L}(z)$ is defined by
		\begin{equation}\label{3.9}
			L(z)\,\widehat{L}(zq^{-1})=\operatorname{qdet} L(z)\,\mathbf 1.
		\end{equation}
		Then the (i,j) entry of the matrix $\widehat{L}(z)$ satisfies
		\begin{equation}\label{3.10}
			\widehat{L}_{ij}(z)
			= (-1)^{\,j-i}\,L_{1\cdots \widehat{i}\cdots N}^{1\cdots \widehat{j}\cdots N}(z),
		\end{equation}
		where
		\begin{equation}
			L_{b_1,\dots,b_m}^{a_1,\dots,a_m}(z)
			= \sum_{p\in S_m} \sgn(p)\,
			L_{a_1,b_{p(1)}}(z)\cdots
			L_{a_m,b_{p(m)}}(zq^{1-m}),
		\end{equation}
		and the hats $\hat{}$ indicate omitted indices.
	\end{lemma}
	
	\begin{proof}
		By the formula of the quantum determinant \eqref{3.8} and the definition of quantum comatrix \eqref{3.9},
		it follows that
		\begin{equation}
			L_1(z)\,L_2(zq^{-1})\cdots L_N(zq^{1-N})\,A_N
			= L_1(z)\,\widehat{L}_1(zq^{-1})\,A_N.
		\end{equation}
		Consequently,
		\begin{equation}\label{3.13}
			L_2(zq^{-1})\cdots L_N(zq^{1-N})\,A_N
			= \widehat{L}_1(zq^{-1})\,A_N.
		\end{equation}
		By applying both sides of \eqref{3.13} to the vector $e_1\otimes e_2\otimes\cdots\otimes e_N$ and comparing the coefficients of $e_i\otimes e_1\otimes\cdots\otimes \widehat{e_j}
		\otimes\cdots\otimes e_N$,
		we obtain \eqref{3.10}.
	\end{proof}
	
	The center elements $\mathfrak{z}(z)$ are related to the quantum determinant by the following Liouville-type
	formula.
	\begin{theorem}
		We have the relations
		\begin{equation}
			\mathfrak{z}(z)= \frac{\operatorname{qdet} L(q^{N-1}z)}{\operatorname{qdet} L(q^{N}z)}.
		\end{equation}
	\end{theorem}
	
	\begin{proof}
		From \eqref{3.7}, we have
		\[
		\mathfrak{z}(q^{-N}z)=\frac{1}{N}tr(L(z)^{-1}L(q^{-N}z)).
		\]
		
		By the definition of $\widehat{L}(z)$ \eqref{3.9},
		\[
		L(z)^{-1} = \bigl(\operatorname{qdet} L(z)\bigr)^{-1}\,\widehat{L}(q^{-1}z).
		\]
		Thus
		\[
		\mathfrak{z}(q^{-N}z)
		= \bigl(\operatorname{qdet} L(z)\bigr)^{-1}\,
		\frac{1}{N} tr\bigl(\widehat{L}(q^{-1}z)L(q^{-N}z)\bigr),
		\]
		where we have used the coefficients of the quantum determinant belongs to the center of the elliptic quantum algebra \( \mathcal{A}_{q,p}(\widehat{\mathfrak{gl}}_N) \).
		
		Then by the formula of the quantum comatrix \eqref{3.10} and simple induction,
		we have
		\[
		tr(\widehat{L}(q^{-1}z)L(q^{-N}z))=N\operatorname{qdet} L(q^{-1}z)\mathbf,
		\]
		thus,
		\[
		\mathfrak{z}(q^{-N}z)
		= \frac{\operatorname{qdet} L(q^{-1}z)}{\operatorname{qdet}L(z)}.
		\]
		By substituting $z$ by $q^Nz$, we finally get
		\[
		\mathfrak{z}(z)
		= \frac{\operatorname{qdet} L(q^{N-1}z)}{\operatorname{qdet} L(q^{N}z)}.
		\]
	\end{proof}
	
	\section{Minor identities for quantum determinants}\label{sec:minor}
	In this section, we extend several identities for quantum minors from Yangians and quantum affine algebras to elliptic quantum algebra $\mathcal{A}_{q,p}(\widehat{\mathfrak{gl}}_N)$\cite{MNO 1996,JZ 2024,JZ 2025, JLZ 2025}.
	
	For an invertible matrix \( A \) with possible non-commutative entries, the quasideterminants \( |A|_{i,j} \) is defined to be \( ((A^{-1})_{j,i})^{-1} \) \cite{KL 1995}. It also equals
	\[
	\begin{vmatrix}
		a_{11} & \dots & a_{1j} & \dots & a_{1N} \\
		\vdots &       & \vdots &       & \vdots \\
		a_{i1} & \dots & \boxed{a_{ij}} & \dots & a_{iN} \\
		\vdots &       & \vdots &       & \vdots \\
		a_{N1} & \dots & a_{Nj} & \dots & a_{NN}
	\end{vmatrix}
	= a_{ij} - r_i^j \left( A^{1,2,\dots,\hat{i},\dots,N}_{1,2,\dots,\hat{j},\dots,N} \right)^{-1} c_j^i,
	\]
	where \( r_i^j \) (\( c_j^i \), respectively) is the row (column, respectively) matrix obtained from the \( i \)-th row (\( j \)-th column, respectively) of \( A \) by deleting the element \( a_{ij} \), and the symbol \( \hat{\ } \) means the index should be omitted.
	
	From the definition of comatrix \eqref{3.9} ,we have the factorization of $\mathrm{qdet}\,L(z)$ .
	\begin{theorem}
		The quantum determinant $\mathrm{qdet}\,L(z)$ admits the factorization
		\begin{align*}
			\mathrm{qdet}\,L(z)
			&=L_{11}(zq^{1-N})\,\bigl|L^{(2)}(zq^{2-N})\bigr|_{22}\cdots \bigl|L^{(N)}(z)\bigr|_{NN} \\
			&= \bigl|L_{(1)}(z)\bigr|_{11}\cdots \bigl|L_{(N-1)}(zq^{2-N})\bigr|_{N-1,N-1}\,L_{NN}(zq^{1-N}),
		\end{align*}
		where for $m=1,\dots,N$, denote by $L^{(m)}(z)$ the submatrix of $L(z)$
		corresponding to the first $m$ rows and columns; the subscript $(m)$ of a matrix
		indicates its submatrix obtained by removing the first $m-1$ rows and columns and $|L(z)|_{ij}$ is the $(i,j)$-th quasideterminant of $L(z)$.
		Moreover, the $N$ factors on the right-hand side of the equality pairwise commute.
	\end{theorem}
	\begin{proof}
		By the definition of comatrix \eqref{3.9} we have
		\begin{equation}\label{4.1}
			\widehat{L}(zq^{-1})=\mathrm{qdet}\,L(z)\,L^{-1}(z).
		\end{equation}
		Taking the $(N,N)$-th entry of \eqref{4.1} we come to
		\begin{equation}\label{4.2}
			\mathrm{qdet}\,L(z)\,(L^{-1}(z))_{NN}=\widehat{L}_{NN}(zq^{-1}).
		\end{equation}
		Then by the definition of quasideterminants, we  get from \eqref{4.2} that
		\begin{equation}\label{4.3}
			\mathrm{qdet}\,L(z)= \mathrm{qdet}\,L^{(N-1)}(zq^{-1}) \, \bigl|L^{(N)}(z)\bigr|_{NN}.
		\end{equation}
		Note that the factors in the decomposition are mutually commutative since the coefficients of quantum determinant belong to the center.
		Then by induction on \eqref{4.3} we obtain the first equality.
		
		The second equality, another decomposition of $\mathrm{qdet}\,L(z)$, can be
		obtained by starting the induction argument from the $(1,1)$-th entry of \eqref{4.1}.
	\end{proof}
	
	Define a new associative algebra over $\mathbb{C}$. Denote this algebra by $\mathcal{B}_N$.
	It is generated by $t_{ij}(n)$, where $1\leq i,j\leq N$ and $n\in\mathbb{Z}$. Let $T(z)=(t_{ij}(z))$ be the generator matrix.
	$T(z)=\sum_{i,j} t_{ij}(z)\otimes E_{ij}$, where $t_{ij}(z)$ are formal series in $z^{-1}$,
	$t_{ij}(z)=\sum_{n\in\mathbb{Z}} t_{ij}(n) z^{-n}$. The defining relation is
	\begin{equation}
		\bigl(\widehat{R}_{12}^*(z/w)\bigr)^{-1} T_1(z) T_2(w) =T_2(w) T_1(z) \bigl(\widehat{R}_{12}(z/w)\bigr)^{-1}.
	\end{equation}
	
	By the same method for getting the quantum determinant of the elliptic algebra $\mathcal{A}_{q,p}(\widehat{\mathfrak{gl}}_N)$, we obtain
	\begin{equation}
		\operatorname{qdet}T(z)=\sum_{\sigma\in S_N} \mathrm{sgn}(\sigma) T_{N,\sigma (N)}(zq^{1-N}) T_{N-1,\sigma(N-1)}(zq^{2-N})\cdots T_{1,\sigma(1)}(z).
	\end{equation}
	
	We also define the quantum minors
	\begin{equation}
		T_{b_1,\dots,b_m}^{a_1,\dots,a_m}(z)=\sum_{\sigma\in S_m} \mathrm{sgn}(\sigma) T_{a_1,b_{\sigma (1)}}(zq^{1-m})\cdots T_{a_m,b_{\sigma (m)}}(z).
	\end{equation}
	
	Clearly we have an algebra isomorphism, which is defined by
	\begin{equation}
		\omega_N: \mathcal{A}_{q,p}(\widehat{\mathfrak{gl}}_N) \to \mathcal{B}_N,\quad L(z) \mapsto T^{-1}(z),
	\end{equation}
	and there is an inverse map of $\omega_N$,
	\begin{equation}
		(\omega_N)^{-1}: \mathcal{B}_N \to \mathcal{A}_{q,p}(\widehat{\mathfrak{gl}}_N),\quad T(z) \mapsto L^{-1}(z).
	\end{equation}
	
	For any $k\geq 0$ introduce the homomorphism $\varphi_k$
	\begin{equation}
		\varphi_k:\mathcal{B}_N \to \mathcal{B}_{N+k}, 
	\end{equation}
	which takes $T_{ij}(z)$ to $T_{k+i,k+j}(z)$. Consider the composition
	\begin{equation}
		\psi_k = (\omega_{N+k})^{-1} \circ \varphi_k \circ \omega_N.
	\end{equation}
	Its action on the generators of $\mathcal{A}_{q,p}(\widehat{\mathfrak{gl}}_N)$ can be expressed in terms of quasideterminants as follows.
	\begin{lemma}\label{Lemma 4.1.}
		For any $1\leq i,j\leq N$, we have
		\[
		\psi_k: L_{ij}(z) \longmapsto
		\begin{pmatrix}
			L_{11}(z) & \cdots & L_{1k}(z) & L_{1,k+j}(z) \\
			\vdots & \vdots & \vdots & \vdots \\
			L_{k1}(z) & \cdots & L_{kk}(z) & L_{k,k+j}(z) \\
			L_{k+i,1}(z) & \cdots & L_{k+i,k}(z) & \boxed{L_{k+i,k+j}(z)}
		\end{pmatrix}.
		\]
	\end{lemma}
	
	The proof of this lemma is almost the same as the case of Yangian (see Lemma 1.11.2 in \cite{Molev 2007}).
	\begin{proof}
		Introduce the sets of indices
		\[
		P = \{1,\dots,k\}, \quad Q = \{k+1,\dots,k+N\}.
		\]
		Consider the block partitioned matrix
		\[
		\begin{bmatrix}
			A_1(z) & B_1(z) \\
			C_1(z) & D_1(z)
		\end{bmatrix}
		\quad \text{and} \quad
		\begin{bmatrix}
			A_2(z) & B_2(z) \\
			C_2(z) & D_2(z)
		\end{bmatrix}.
		\]
		For any matrix $X$, we will denote by $X_{PQ}$ the submatrix whose rows and columns are enumerated by the elements of the sets $P$ and $Q$ respectively, whose entries are the series $L_{ij}(z)$ and $t_{ij}(z)$ so that
		\[
		\begin{aligned}
			A_1(z) &= T(z)_{PP}, & B_1(z) &= T(z)_{PQ}, & C_1(z) &= T(z)_{QP}, \text{and} & D_1(z) &= T(z)_{QQ}, \\
			A_2(z) &= L(z)_{PP}, & B_2(z) &= L(z)_{PQ}, & C_2(z) &= L(z)_{QP}, \text{and} & D_2(z) &= L(z)_{QQ}.
		\end{aligned}
		\]
		Let
		\[
		\begin{bmatrix}
			A_1(z) & B_1(z) \\
			C_1(z) & D_1(z)
		\end{bmatrix}^{-1}
		=
		\begin{bmatrix}
			A_1'(z) & B_1'(z) \\
			C_1'(z) & D_1'(z)
		\end{bmatrix}
		\quad \text{and} \quad
		\begin{bmatrix}
			A_2(z) & B_2(z) \\
			C_2(z) & D_2(z)
		\end{bmatrix}^{-1}
		=
		\begin{bmatrix}
			A_2'(z) & B_2'(z) \\
			C_2'(z) & D_2'(z)
		\end{bmatrix}.
		\]
		By its definition, the homomorphism $\psi_k$ maps  $(\omega_N)^{-1}(T(z))=L^{-1}(z)$ to $D_2'(z)$:
		\[
		\begin{aligned}
			\psi_k(L^{-1}(z)) &= (\omega_{N+k})^{-1} \circ \varphi_k \circ \omega_N(L^{-1}(z)) \\
			&= (\omega_{N+k})^{-1} \circ \varphi_k(T(z)) \\
			&= (\omega_{N+k})^{-1} (D_1(z)) \\
			&=D_2'(z)
		\end{aligned}
		\]
		where we know
		\[
		D_2'(z) = \bigl( D_2(z) - C_2(z) A_2^{-1}(z) B_2(z) \bigr)^{-1}.
		\]
		Hence,
		\[
		\psi_k: L(z) \longmapsto D_2(z) - C_2(z) A_2^{-1}(z) B_2(z).
		\]
		Taking the $ij$-th entry and using the definition of the quasideterminant we get the desired formula for the image of $L_{ij}(z)$.
	\end{proof}
	
	For any two sets $I=\{i_1,\ldots,i_r\}$, $J=\{j_1,\ldots,j_s\}$ with
	$i_1<\cdots<i_r$ and $j_1<\cdots<j_s$, we denote by $l(I,J)$ the inversion
	number of the sequence $i_1,\ldots,i_r,j_1,\ldots,j_s$.
	Let $A$ be any square matrix of size $N\times N$.
	For any subsets $I=\{i_1,\ldots,i_k\}$, $J=\{j_1,\ldots,j_k\}$ of $[1,N]$,
	we denote by $A^J_I$ the matrix whose $ab$-th entry is $A_{i_a j_b}$.
	For any subsets $I=\{i_1<i_2<\cdots<i_k\}$, we denote by $A_I$ the submatrix
	of $A$ with rows and columns indexed by $I$.
	The following theorem is an analog of Jacobi's ratio theorem for the quantum determinants \cite{JZ 2025,JLZ 2025,KL 1995}.
	\begin{theorem}\label{Jacobi's ratio theorem}
		Let $I=\{i_1<\cdots<i_k\}$ and $J=\{j_1<\cdots<j_k\}$ be two subsets of $[1,N]$
		of the same cardinality, and $I^c=\{i_{k+1}<\cdots<i_N\}$ and $J^c=\{j_{k+1}<\cdots<j_N\}$
		be their complements. Then
		\begin{equation} \label{4.8}
			L_{j_{k+1},\dots,j_N}^{i_{k+1},\dots,i_N}(zq^{-k}) = \operatorname{qdet}(L(z)) (\omega_N)^{-1}(T^{j_k,\dots,j_1}_{i_k,\dots,i_1}(z)) =\operatorname{qdet}(L(z))
			((L^{-1})^{j_k,\dots,j_1}_{i_k,\dots,i_1}(z)).
		\end{equation}
	\end{theorem}
	\begin{proof}
		Multiplying $(L_1(z))^{-1}, \ldots, (L_k(zq^{1-k}))^{-1}$ on the left of the formula \eqref{qdet}, we get that
		\[
		L_{k+1}(zq^{-k})\cdots L_N(zq^{1-N})A_N
		= \operatorname{qdet}L(z)\, L_k(zq^{1-k})^{-1}\cdots L_1(z)^{-1} A_N.
		\]
		Applying both sides to the vector $e_1 \otimes e_2 \otimes \cdots \otimes e_N$,
		\[
		\sum_{a_{k+1},\ldots,a_N}\;\sum_{\sigma\in S_N}\sgn(\sigma)\,
		L_{a_{k+1},\sigma(k+1)}(zq^{-k})\cdots L_{a_N,\sigma(N)}(zq^{1-N})
		\otimes e_{\sigma(1)} \otimes \cdots \otimes e_{\sigma(k)} \otimes e_{a_{k+1}} \otimes \cdots \otimes e_{a_N}
		\]
		\[
		= \operatorname{qdet}L(z)\,
		\sum_{b_1,\ldots,b_k}\;\sum_{\sigma\in S_N}\sgn(\sigma)\,
		L_{b_k,\sigma(k)}(zq^{1-k})^{-1}\cdots L_{b_1,\sigma(1)}(z)^{-1}
		\otimes e_{b_1} \otimes \cdots \otimes e_{b_k} \otimes e_{\sigma(k+1)} \otimes \cdots \otimes e_{\sigma(N)},
		\]
		then comparing the coefficient of
		$e_{j_1}\otimes \cdots \otimes e_{j_k} \otimes e_{i_{k+1}} \otimes \cdots \otimes e_{i_N}$,
		\[
		\sum_{\sigma\in S_{N-k}}\sgn(\sigma)\,
		L_{i_{k+1},j_{\sigma(k+1)}}(zq^{-k})\cdots
		L_{i_N,j_{\sigma(N)}}(zq^{1-N})
		\otimes e_{j_1}\otimes \cdots \otimes e_{j_k}\otimes e_{i_{k+1}}\otimes \cdots \otimes e_{i_N}
		\]
		\[
		= \operatorname{qdet}L(z)\,
		\sum_{\sigma\in S_k}\sgn(\sigma)\,
		L_{j_k,i_{\sigma(k)}}(zq^{1-k})^{-1}\cdots
		L_{j_1,i_{\sigma(1)}}(z)^{-1}
		\otimes e_{j_1}\otimes \cdots \otimes e_{j_k}\otimes e_{i_{k+1}}\otimes \cdots \otimes e_{i_N},
		\]
		we obtain
		\[
		\sum_{\sigma\in S_{N-k}}\sgn(\sigma)\,
		L_{i_{k+1},j_{\sigma(k+1)}}(zq^{-k})\cdots
		L_{i_N,j_{\sigma(N)}}(zq^{1-N})
		\]
		\[
		= \operatorname{qdet}L(z)\,
		\sum_{\sigma\in S_k}\sgn(\sigma)\,
		L_{j_k,i_{\sigma(k)}}(zq^{1-k})^{-1}\cdots
		L_{j_1,i_{\sigma(1)}}(z)^{-1}.
		\]
		If we write this in the form of quantum minors, we get \eqref{4.8}.
	\end{proof}
	
	As a special case with $I=J=N$,  we have the following corollary.
	\begin{corollary}
		\begin{equation} \label{4.9}
			\operatorname{qdet}(L(z))
			(\operatorname{qdet}(L^{-1}(z)))=1.
		\end{equation}
	\end{corollary}
	
	The action of $\psi_k$ on the quantum minors is described in the following theorem.
	\begin{theorem}
		We have
		\begin{equation}\label{4.10}
			\psi_k: \, L^{a_1\cdots a_m}_{b_1\cdots b_m}(z) \longmapsto
			\bigl( L^{1\cdots k}_{1\cdots k}(zq^{-m}) \bigr)^{-1} \cdot
			L^{1\cdots k,\, k+a_1,\, \dots,\, k+a_m}_{1\cdots k,\, k+b_1,\, \dots,\, k+b_m}(z),
		\end{equation}
		where $a_i,b_i \in \{1,\dots,N\}$.
	\end{theorem}
	\begin{proof}
		Apply the homomorphism $\omega_N$ to \eqref{4.8} we have
		\begin{equation}\label{4.11}
			\omega_N\Bigl( L^{i_{k+1},\dots,i_N}_{j_{k+1},\dots,j_N}(zq^{-k}) \Bigr)
			= \omega_N\Bigl( \operatorname{qdet}L(z) \,\Bigr) T^{j_k,\dots,j_1}_{i_k,\dots,i_1}(z) ,
		\end{equation}
		and then apply $\omega_N$ to \eqref{4.9}, we have $\omega_N (\qdet L(z)) \qdet T(z)=1$,
		thus when we take ${i_{k+1},\dots,i_N}={a_1,\cdots a_m}$ and ${j_{k+1},\dots,j_N}={b_1,\cdots, b_m}$, \eqref{4.11} becomes
		\begin{equation}\label{4.12}
			\omega_N\Bigl( L^{a_1\cdots a_m}_{b_1\cdots b_m}(z) \Bigr)
			= \bigl( \operatorname{qdet}T(zq^{N-m}) \bigr)^{-1}
			\cdot T^{b_{m+1},\dots,b_N}_{a_{m+1},\dots,a_N}(zq^{N-m}),
		\end{equation}
		then apply $\varphi_k$ to \eqref{4.12},
		\begin{equation}\label{4.13}
			\varphi_k\Bigl( \omega_N\bigl( L^{a_1\cdots a_m}_{b_1\cdots b_m}(z) \bigr) \Bigr)
			=(T^{k+1,\dots,k+N}_{k+1,\dots,k+N}(zq^{N-m}))^{-1}
			\cdot T^{k+b_{m+1},\dots,k+b_N}_{k+a_{m+1},\dots,k+a_N}(zq^{N-m}),
		\end{equation}
		finally, apply $(\omega_{N+k})^{-1}$ to \eqref{4.13}, we obtain
		\begin{equation}\label{4.14}
			\begin{aligned}
				\psi_k\Bigl( L^{a_1\cdots a_m}_{b_1\cdots b_m}(z) \Bigr)
				&= (\omega_{N+k})^{-1}\Bigl(
				(T^{k+1,\dots,k+N}_{k+1,\dots,k+N}(zq^{N-m}))^{-1}
				\cdot \,
				T^{k+b_{m+1},\dots,k+b_N}_{k+a_{m+1},\dots,k+a_N}(zq^{N-m})
				\Bigr) \\
				&= \bigl( L^{1\cdots k}_{1\cdots k}(zq^{-m}) \bigr)^{-1} \cdot
				L^{1\cdots k,\, k+a_1,\, \dots,\, k+a_m}_{1\cdots k,\, k+b_1,\, \dots,\, k+b_m}(z),
			\end{aligned}
		\end{equation}
		where we have used the equation \eqref{4.8} again.
	\end{proof}
	
	Take $m=1$, we have the following corollary.
	\begin{corollary}\label{Corollary 4.5.}
		For any $1\leq i,j\leq N$, we have
		\begin{equation}\label{4.15}
			\psi_k: \, L_{ij}(z) \longmapsto
			\bigl( L^{1\cdots k}_{1\cdots k}(zq^{-1}) \bigr)^{-1}
			L^{1\cdots k,\, k+i}_{1\cdots k,\, k+j}(z).
		\end{equation}
	\end{corollary}
	
	For any $1 \leq i,j \leq N$, introduce the following series with coefficients in $\mathcal{A}_{q,p}(\widehat{\mathfrak{gl}}_{N+k})$
	\[
	L_{ij}^{\#}(z) = L^{1\cdots k,\, k+i}_{1\cdots k,\, k+j}(z)
	\]
	and combine them into the matrix $L^{\#}(z) = [L_{ij}^{\#}(z)]$. Let $L^{(k)}(z)$ be the submatrix of $L(z)$ determined by the first $k$ rows and columns.
	The following is an analog of Sylvester's theorem for the quantum determinants of the elliptic quantum algebra
	$\mathcal{A}_{q,p}(\widehat{\mathfrak{gl}}_N)$ \cite{Molev 2007}.
	
	\begin{theorem}
		The mapping
		\begin{equation}\label{4.16}
			L_{ij}(z) \longmapsto L_{ij}^{\#}(z), \quad 1 \leq i,j \leq N,
		\end{equation}
		defines a homomorphism $\mathcal{A}_{q,p}(\widehat{\mathfrak{gl}}_N) \to \mathcal{A}_{q,p}(\widehat{\mathfrak{gl}}_{N+k}).$
		Moreover, we have the identity
		\begin{equation}\label{4.17}
			\mathrm{qdet}\, L^{\#}(z)
			= \mathrm{qdet}\, L(z) \cdot \mathrm{qdet}\, L^{(k)}(zq^{-1}) \cdots \mathrm{qdet}\, L^{(k)}(zq^{1-N}).
		\end{equation}
	\end{theorem}
	\begin{proof}
		By \eqref{4.15}, we have
		\begin{equation}\label{4.18}
			\psi_k: L_{ij}(z) \longmapsto \bigl( \mathrm{qdet}\, L^{(k)}(zq^{-1}) \bigr)^{-1} \cdot L_{ij}^{\#}(z).
		\end{equation}
		Since $\psi_k$ is a homomorphism and $\left[ L_{a_i b_j}(z), L^{a_1\cdots a_m}_{b_1\cdots b_m}(w) \right] = 0$, $i,j \in \{1,\dots,m\}$ (since we can regard the coefficients of $L_{a_i b_j}(z)$, $i,j \in \{1,\dots,m\}$ generate a  subalgebra of $\mathcal{A}_{q,p}(\widehat{\mathfrak{gl}}_N)$, where $L_{b_1\cdots b_m}^{a_1\cdots a_m}(z)$ belong to the center of the subalgebra). Thus we have that the mapping \eqref{4.16} defines a homomorphism. Furthermore, by \eqref{4.10},
		\begin{equation}\label{4.19}
			\psi_k: L^{1\cdots N}_{1\cdots N}(z) \longmapsto \bigl( \mathrm{qdet}\, L^{(k)}(zq^{-N}) \bigr)^{-1} \cdot \mathrm{qdet}\, L(z).
		\end{equation}
		On the other hand, expanding the quantum minors, we obtain from \eqref{4.18} that
		\begin{equation}\label{4.20}
			\psi_k: L^{1\cdots N}_{1\cdots N}(z) \longmapsto
			\bigl( \mathrm{qdet}\, L^{(k)}(zq^{-1}) \cdots \mathrm{qdet}\, L^{(k)}(zq^{-N}) \bigr)^{-1}
			\cdot \mathrm{qdet}\, L^{\#}(z),
		\end{equation}
		thus, comparing the image of \eqref{4.19} and \eqref{4.20}, we get \eqref{4.17}.

	\end{proof}
	
	\bigskip
\centerline{\bf Acknowledgments}
\medskip

The work is supported in part by the National Natural Science Foundation of China (grant no. 12571026).
			
			\bibliographystyle{amsalpha}

		\end{document}